\documentclass[reqno,11pt]{amsart} 

\usepackage{amsmath, amsthm, a4, latexsym, amssymb, setspace, mathrsfs}
\usepackage[dvips]{graphicx}

\allowdisplaybreaks

 \, \textheight  620pt 
 \, \textwidth 150mm
\def\@rmrk#1#2{\refstepcounter
    {#1}\@ifnextchar[{\@yrmrk{#1}{#2}}{\@xrmrk{#1}{#2}}}

\makeatletter\@addtoreset{equation}{section}\makeatother

 \newfont{\bfit}{cmbxti10 scaled 1200}

 \newcommand{\eps}{\varepsilon}

 \newcommand{\ball}{{\mathcal B}}
 \newcommand{\R}{\mathbb{R}}

 \newcommand{\prob}{\mathbb{P}}

 \newcommand{\1}{{\sf 1}}

 \newcommand{\skrib}{{\mathcal B}}

 \newcommand{\heap}[2]{\genfrac{}{}{0pt}{}{#1}{#2}}
 \newcommand{\sfrac}[2]{ \, \mbox{$\frac{#1}{#2}$}}

\allowdisplaybreaks

\renewcommand{\subsection}{\secdef \subsct\sbsect}
\newcommand{\subsct}[2][default]{\refstepcounter{subsection}
\vspace{0.15cm}
{\flushleft\bf \arabic{section}.\arabic{subsection}~\bf #1  }
\nopagebreak\nopagebreak}
\newcommand{\sbsect}[1]{\vspace{0.1cm}\noindent
{\bf #1}\vspace{0.1cm}}

{\nopagebreak {\hfill\rule{2mm}{2mm}}\\ }

\newtheorem{theorem}{Theorem}[section]

\newtheorem{lemma}[theorem]{Lemma}

\newtheorem{prop}[theorem]{Proposition}

\renewcommand{\P}{\mathbb{P}}

\newcommand{\E}{\mathbb{E}}

\newtheoremstyle{thm}{1.5ex}{1.5ex}{\itshape\rmfamily}{}
{\bfseries\rmfamily}{}{2ex}{}

\newtheoremstyle{rem}{1.3ex}{1.3ex}{\rmfamily}{}
{\itshape\rmfamily}{}{1.5ex}{}
\theoremstyle{rem}

\refstepcounter{subsubsection}

\def\thebibliography#1{\section*{Bibliography}
  \list%
  {\arabic{enumi}.}
    {\settowidth\labelwidth{[#1]}\leftmargin\labelwidth
    \advance\leftmargin\labelsep
    \parsep0pt\itemsep0pt
    \usecounter{enumi}}
    \def\newblock{\hskip .11em plus .33em minus .07em}
    \sloppy                   
    \sfcode`\.=1000\relax}

\begin{document}

\title[The most visited sites of planar Brownian motion]
{On the most visited sites of planar Brownian motion}

\centerline{\LARGE \bf On the most visited sites of} 
\ \\[-4mm]
\centerline{\LARGE \bf planar Brownian motion}

\vspace{0.5cm}

\thispagestyle{empty}
\vspace{0.2cm}
\centerline{\textsc{Valentina Cammarota}\footnote{
 Department of Statistics,
``Sapienza" University of Rome,
P.\,le Aldo Moro, 5 -- 00185 Rome, Italy.\\
E--mail:  \, \texttt{valentina.cammarota@uniroma1.it}.}
\,   and  \, \textsc{Peter M\"orters}\footnote{
Department of Mathematical Sciences,
University of Bath, Bath BA2 7AY, England.\\
E--mail:  \, \texttt{maspm@bath.ac.uk}.}}

\vspace{0.3cm}


\begin{quote}{\small {\bf Abstract:} Let $(B_t \colon t \ge 0)$ be a planar Brownian motion
and define gauge functions $\phi_\alpha(s)=\log(1/s)^{-\alpha}$ for $\alpha>0$. If $\alpha<1$ we show that
almost surely there exists a point~$x$ in the plane such that ${\mathcal H}^{\phi_\alpha}(\{t \ge 0 \colon B_t=x\})>0$,
but if $\alpha>1$ almost surely ${\mathcal H}^{\phi_\alpha} (\{t \ge 0 \colon B_t=x\})=0$ simultaneously for all $x\in\R^2$. 
This  resolves a longstanding open problem posed by S.\,J.~Taylor in 1986.}
\end{quote}
\vspace{0.5cm}

{\footnotesize
{\bf MSc classification (2010):} 60J65.\\[-5mm] 

}



\section{Introduction and statement of main results}

Let $(B_t \colon t \ge 0)$ be a standard planar Brownian motion.  Dvoretzky, Erd\H{o}s and Kakutani~(1958) first showed 
that, almost surely, there exist points~$x$ in the plane such that $\{t \ge 0 \colon B_t=x\}$, the set of times where
the Brownian path visits~$x$, is uncountably infinite.  Modern proofs of this fact are given in Le Gall~(1987) 
and M\"orters and Peres~(2010). The result naturally raises the question: {How large can the sets $\{t \ge 0 \colon B_t=x\}$ 
be? Kaufman's famous dimension-doubling 
theorem  
implies that, almost surely, for all points~$x$ in the plane this set has Hausdorff dimension zero.%
\smallskip

S.~James Taylor in his influential survey paper `The measure theory of random fractals'
of 1986 raises the problem in terms of Hausdorff gauge functions. Letting
$\phi\colon(0,\eps) \to [0,1]$ be a right continuous, increasing function with
$\phi(0+)=0$,  we denote 
$${\mathcal H}^\phi(E)=\lim_{\delta\downarrow 0} \inf\Big\{ \sum_{i=1}^\infty \phi({\rm diam}(E_i)) \colon E \subseteq \bigcup_{i=1}^\infty E_i, 
0<{\rm diam}(E_i)<\delta\Big\},$$ 
the $\phi$-Hausdorff measure of the set $E\subset\R$. 
Problem~5 in Taylor~(1986) is the following question:

\begin{quote}Which gauge functions~$\phi$ are such that, almost surely, we have
${\mathcal H}^\phi(\{t\geq 0 \colon B_t=x\})=0$ for all $x\in\R^2$?
\end{quote}

In the paragraph following this question Taylor focuses on the gauge functions of the form $\phi_\alpha(s)=\log(1/s)^{-\alpha}$ for $\alpha>0$, moving to
a weaker but still challenging form of the problem. He notes that the results of Perkins and Taylor~(1987) imply that 
these functions satisfy the above condition for $\alpha>2$
and states that 

\begin{quote}It is not much more than a guess, but my hunch is that $\phi_\alpha$ satisfies the condition for $\alpha>1$,
but not for $0<\alpha<1$.
\end{quote}

The weaker form of the problem is reiterated in Perkins and Taylor~(1987) as Problem~3.11. More than fifteen years
later Xiao~(2004) notes that the problem is still open even in the weaker form. The problem has recently been reiterated as
Problem~5 in the open problems section of M\"orters and Peres~(2010). It is the aim of the present paper to solve the weaker form of the 
problem and to confirm Taylor's conjecture.
\medskip

Although the solution involves both a new lower and upper bound, only the lower bound involves substantial work.
It relies on the construction of an intersection local time for points of infinite multiplicity, 
which is due to Bass, Burdzy and Khoshnevisan~(1994). Denoting by $\ball(x,\eps)\subset\R^2$ the open disc with centre~$x$ 
and radius~$\eps$ and by $\ball$ the open disc $\ball(0,1)$ we let $N^x_\eps$ be the number of times the Brownian path $(B_t \colon t\geq 0)$ travels from the 
point~$x \in \skrib$ to the circle~$\partial \ball(x,\eps)$ before it leaves $\ball$ for the first time. Naturally for
most points $x$ we have $N^x_\eps=0$, but some points satisfy
\begin{equation}\label{lt=a}
\lim_{\eps\downarrow 0} \frac{N^x_\eps}{\log(1/\eps)}=a,
\end{equation}

for some $a>0$. If~\eqref{lt=a} holds we say that Brownian motion `spends $a$ units of local time at~$x$'.
Note that the points~$x$ are by definition points of infinite multiplicity. Bass, Burdzy and Khoshnevisan~(1994)
in their Theorem~1.1 show that, for $0<a<\frac12$ there exists a (random) measure~$\beta_a$ on the plane such that~\eqref{lt=a}
holds for $\beta_a$-almost every~$x$. Moreover the measure is nonzero and has carrying Hausdorff dimension $2-a$ almost surely.
\medskip

Our main result shows that the points selected according to $\beta_a$ also have a large inverse image under
the Brownian motion, thus providing a lower bound for Taylor's Problem~5. 
\medskip

\begin{theorem}\label{lower}
Let $0<\alpha<\frac12$ and take $\beta_a$ to be the measure on points with local time~$a$. 
Define the gauge function
$$\varphi(\eps):= \frac{\log\log\log(1/\eps)}{\log (1/\eps)}.$$
Then, almost surely, 
$${\mathcal H}^\varphi\big(\{t\geq 0 \colon B_t=x\}\big)\geq a\,\sqrt{\sfrac\pi2}>0,$$
for $\beta_a$-almost every $x$.
\end{theorem}

The proof of Theorem~\ref{lower} is given in Section~2. To solve the weaker form of Taylor's Problem~5
we also need an upper bound confirming that the bound above has the right power of the logarithm. 
\smallskip

\begin{theorem} \label{upper}
For every gauge function $\phi$ with 
$\phi(\eps) \log (1/\eps) \to 0$, almost surely, 
$${\mathcal H}^\phi\big(\{t\geq 0 \colon B_t=x\}\big)=0\qquad \mbox{for every $x\in\R^2$.} $$
\end{theorem}

The proof of Theorem~\ref{upper}, which uses a simple first moment estimate, is given in Section~3. Combining
Theorems~\ref{lower} and~\ref{upper} confirms Taylor's conjecture on the weaker form of Problem~5
as stated in the abstract. The strong form of the problem remains open at this point, see Section~4 for 
a new conjecture and some further remarks on this.

\section{Proof of the lower bound}

Let $\prob^z$ and $\E^z$ be the distribution and the corresponding expectation of planar Brownian motion 
started at~$z\in\R^2$, and let $\tau(A)$ be the first hitting time of a Borel set $A\subset\ball$. In
particular we let $\tau=\tau(\partial \skrib)$ be the first hitting time of the unit disc and denote by 
$p^{\skrib}_t(\cdot\,,\cdot\,)$ the transition (sub-)density for 
the Brownian motion killed at $\tau$.
\medskip

Let $B=(B_t \colon 0 \le t \le \tau)$ be the Brownian motion started at the
origin and killed upon leaving~$\ball$. The  idea of the proof of Theorem~\ref{lower}
is to use the representation of $B$ as seen from a typical point chosen accordingly to $\beta_a$, which is given in Section~5
of Bass et al.~(1994). The main ingredient of this representation is the construction of the process $(Z^x_a(t) \colon 0\leq t \leq \tau)$ which we now recall.
\medskip

Fix some $x \in \skrib$ and let $h$ be a strictly positive harmonic function in $\skrib \setminus \{x\}$ with zero boundary values on $\partial \skrib$ and a pole at $x$ such that 
\begin{eqnarray} \label{h} \lim_{z\to x} \frac{h(z)}{|\log|z-x||} =1.\end{eqnarray}
The $h$-transform of $B$ is a Markov process in $\skrib \setminus \{x\}$ with transition density $p^h_t$ given by 
\begin{eqnarray} \label{p_h} p^h_t(z,y)= \frac{h(y)}{h(z)}\, p^\skrib_t(z,y). \end{eqnarray}
The distribution of this process started in $y\in\ball \setminus \{x\}$ is denoted $\P^y_h$. Let $C_*[0,\infty)$ be the set of all paths $e\colon [0, \infty) \to \ball \cup \{ \Delta\}$  
which are continuous on some interval $[0,\sigma)$ and then jump to the isolated coffin state $\Delta$. The canonical process on $C_*[0,\infty)$ will be denoted by~$X$, 
i.e. $X_t(e)=e_t$ for all  $e\in C_*[0,\infty)$ and $t\geq 0$. 
There is an, up to a constant factor unique, positive and $\sigma$-finite measure $H^x$ on $C_*[0,\infty)$ such that 
\begin{itemize}
\item $\displaystyle\lim_{t\downarrow 0} X_t =x,  \mbox{ $H^x$-almost surely,}$\\[-1mm]
\item $H^x$ is strong Markov for the transition densities $p^{h}_t(\cdot\,,\cdot\,)$,\\[-1mm]
\item $H^x(\tau(A)<\infty)<\infty$, for any compact set $A\subset \ball\setminus\{x\}$.
\end{itemize}
The excursion law $H^x$ will be normalised so that\\[-1mm]
\begin{equation}\label{L5.1(i)} \lim_{\eps \downarrow 0} \frac{H^x (\tau(\partial \skrib(x,\eps))<\infty)}{|\log \eps|}=a, {\phantom{\Bigg\{ } }\end{equation}
see  Lemma 5.1 in Bass et al.~(1994) for existence of this normalisation, and Burdzy~(1987) for background on excursion laws.%
\medskip%

Let $\mbox{Leb}$ denote the Lebesgue measure on $[0,\infty)$ and $Y$ be a Poisson point process on $[0,\infty)\times C_*[0,\infty)$ with mean measure 
$\mathrm{Leb} \otimes H^x$. Bass et al.~(1994), Lemma 5.1, show that, for every $t< \infty$, 
\begin{eqnarray} \label{lifetime} \sum_{\heap{(s,e_s) \in Y}{ s < t}}\sigma(e_s) < \infty, \end{eqnarray}
almost surely, where $\sigma(e_s)$ is the lifetime of the excursion $e_s$.
\medskip

\pagebreak[3]

The trajectories of $Z^x_a$ are assembled from three parts,
\begin{equation*}
Z^x_a(t)=\left\{
\begin{array}{ll} Z_1(t) &\mbox{ if } 0 \le t \le t_1,\\
Z_2(t-t_1) & \mbox{ if }  t_1 \le t \le t_1+t_2,\\
Z_3(t-t_1-t_2) & \mbox{ if } t_1+t_2 \le t \le\tau= t_1+t_2+t_3.
\end{array}
\right.
\end{equation*}
\begin{enumerate}
\item $(Z_1(t) \colon 0 \le t \le t_1)$ is an $h$-process  in $\skrib \setminus \{x\}$ which starts from $0$ and is stopped when it approaches $x$ at time $t_1$.\\[-1mm]

\item For $u>0$ let 
$$T(u)=\sup \left\{t \colon \sum_{\heap{(s,e_s) \in Y}{ s < t}}\sigma(e_s) \le u\right\}.$$
By (\ref{lifetime}), $T(u)$ is well defined for all $u < \infty$ almost surely. Note that, almost surely, $T(u)< \infty$ for each $u$ and,
for almost all $u$, there is a point $(s,e_s)\in Y$ such that $s=T(u)$. For such $u$ we define
$$Z_2(u)=e_{T(u)} \left( u-\sum_{\heap{(s,e_s) \in Y}{ s < T(u)}} \sigma(e_s) \right)$$
and for the remaining $u$ we define $Z_2(u)=x$. 
Let $t_2=\sum_{s<1}\sigma(e_s)$ and observe that this defines a continuous process $(Z_2(t) \colon 0 \le t \le t_2)$.  \\[-1mm]

\item $(Z_3(t) \colon 0 \le t \le t_3)$ is a Brownian motion starting from $x$ and killed at the first exit from the unit ball~$\ball$.\\[-1mm]
\end{enumerate}
We assume that $Z_1$, $Z_2$ and $Z_3$ are independent. The distribution of the process $(Z^x_a(t) \colon 0 \le t \leq \tau)$ will be denoted 
by~${\mathbb Q}^x_a$. 

The process $Z^x_a$  may be thought of as a Brownian motion conditioned to spend $a$ units of local time at $x$. It is possible to interpret 
${\mathbb Q}^x_a$ as the distribution of $B$ conditioned by the event that $x$ is in the support of $\beta_a$. 
This follows from the   `Palm measure' decomposition of $\beta_a$ stated below and proved in Bass et al.~(1994), Theorem~5.2.
\smallskip

\begin{lemma} \label{Palm}
For every $a \in (0,\frac 1 2)$ and every nonnegative measurable function $f$ on $\mathbb{R}^2 \times C_*[0,\infty)$ we have 
$${\mathbb E}^0 \int f (y, B) \beta_a(d y)=\int_\skrib \int \1_{f(y,B)}  |1-y^2|^a |\log |y|| \; {\mathbb Q}_a^y(d B)\; dy.$$
\end{lemma}

Let $L^x_a$ be the right-continuous generalised inverse of the sum of the excursion lengths of $Z^x_a$ from $x$,
 $$L^x_a(t): =\sup \Bigl\{\theta \colon   \sum_{\heap{(s,e_s) \in Y}{ s < \theta}}\sigma(e_s)   \le t \Bigr\}.$$
  $L^x_a$ is the local time (in the sense of excursion theory) of $Z^x_a$ at $x$. The following proposition is the main ingredient of the proof of the lower bound. 
\smallskip

\begin{prop}\label{LILL}
For every $a \in (0, \frac 1 2)$ and $x \in \mathbb{R}^2$, we have 
\begin{equation}\label{LIL}
\limsup_{\eps \downarrow 0} \frac{L^x_a(t_1+\eps)}{\varphi(\eps)}= a^{-1}\, \sqrt{\sfrac2\pi} < \infty, \qquad  \mathbb{Q}^x_a\mbox{-almost surely,}
\end{equation}
where $\varphi(\eps):=\frac{\log \log \log(1/\eps)}{\log (1/\eps)}$.
\end{prop}

Before proving Proposition \ref{LILL}, we show how this implies Theorem \ref{lower}.  

First, the natural link between the Hausdorff  measure   ${\mathcal H}^\varphi\big(\{t\geq 0 \colon B_t=x\}\big)$ and the law of the iterated logarithm in (\ref{LIL})  
is the Rogers-Taylor theorem stated below,  see M\"orters and Peres~(2010), Proposition 6.44, for a proof.
 
\begin{lemma} \label{R-T}
Let $\mu$ be a Borel measure on $\mathbb{R}$ and let $\phi$ be a Hausdorff gauge function. If $\Lambda \subset \mathbb{R}$ is a closed set and 
$$A:= \left\{ t \in \Lambda:  \limsup_{\eps \downarrow 0} \frac{\mu[t,t+\eps]}{\phi(\eps)}<\alpha  \right\},$$
then ${\mathcal H}^{\phi}(A) \ge \frac{\mu(A)}{\alpha}$.
\end{lemma} 
 
In our case we consider the closed set  $\Lambda_x:=\{t \ge 0: Z^x_a(t)=x\}$, the (probability)
measure $\ell_a^x$ given by $\ell_a^x[t,t+\eps]:=L_a^x(t+\eps)-L_a^x(t)$, the gauge function  $\varphi$ and the sets 
$$A_\alpha:=\left\{t \in \Lambda_x: \limsup_{\eps \downarrow 0} \frac{\ell_a^x[t,t+\eps]}{\varphi(\eps)} < \alpha\right\}.$$ 
Then, by  Lemma \ref{R-T}, we have ${\mathcal H}^{\varphi}(\Lambda_x)\ge {\mathcal H}^{\varphi}(A_\alpha)  \ge \frac 1 \alpha \; {\ell_a^x(A_\alpha)}$,
and we now show that,  $\mathbb{Q}_a^x$-almost surely, $\ell_a^x(A_\alpha)=1$ for suitable $\alpha>0$.
\medskip

For $0\leq t\leq 1$ let $T_t:=\inf \{s>0\colon L_a^x(s)\geq t\}<\infty$. 
By construction the process $(Z^x_a(T_t+s) \colon 0\leq s \leq T_1-T_t)$ has the same law
as $(Z^x_a(t_1+s) \colon 0\leq s \leq T_{1-t})$. Hence it follows from Proposition~\ref{LILL} that, for all $t \in (0, 1)$, 
$$\limsup_{\eps \downarrow 0} \frac{\ell_a^x[T_t, T_t+\eps]}{\varphi(\eps)} = a^{-1}\, \sqrt{\sfrac2\pi}
, \qquad  \mathbb{Q}^x_a\mbox{-almost surely.}$$
By applying Fubini's theorem we get 
$$\mathrm{Leb}\bigg\{ t \in (0, 1)\colon \limsup_{\eps \downarrow 0} \frac{\ell_a^x[T_t, T_t+\eps]}{\varphi(\eps)} = a^{-1}\, \sqrt{\sfrac2\pi}  \bigg\}=1,  \qquad  \mathbb{Q}^x_a\mbox{-almost surely,}$$
that is 
$$\ell_a^x \bigg\{s \in \Lambda_x \colon \limsup_{\eps \downarrow 0} \frac{\ell_a^x[s, s+\eps]}{\varphi(\eps)}= a^{-1}\, \sqrt{\sfrac2\pi} \bigg\}=1, \qquad  \mathbb{Q}^x_a\mbox{-almost surely.}$$
Hence ${\ell_a^x(A_\alpha)}=1$ for all $\alpha>a^{-1}\, \sqrt{2/\pi}$,
and so ${\mathcal H}^{\varphi}(\Lambda_x) \geq a\, \sqrt{\pi/2}>0$, $\mathbb{Q}_a^x$-almost surely.%
\medskip%

To complete the proof of the lower bound we apply Lemma~\ref{Palm} to the nonnegative, measurable function 
$f(x, B)= \1_{\{(x,B)\colon{\mathcal H}^{\varphi}(\Lambda_x)<a\sqrt{\frac\pi2}\}}$. We have seen that
$$\int_\ball \int f(x, B) \; |1-x^2|^a |\log |x|| \; \mathbb{Q}_a^x(d B) \, dx=0,$$
and this implies that 
$$\mathbb{E}^0 \int f(x,B) \, \beta_a(d x)=0,$$
that is ${\mathcal H}^{\varphi}(\Lambda_x)\geq a\, \sqrt{\pi/2}$ for $\beta_a$-almost every $x$, $\mathbb{P}^0$-almost surely,
as required to prove Theorem~\ref{lower} subject to the proof of Proposition~\ref{LILL}.
\bigskip

We now prove Proposition \ref{LILL}. Bertoin and Caballero~(1995), making use of the work of Fristedt and Pruitt~(1971), have shown the 
following result, see their Theorem~1.
\smallskip

\begin{lemma} \label{F-P} 
Let $\xi$ be a subordinator, $\Phi$ its Laplace exponent and $S$ its  right-continuous generalised inverse. If $\Phi$ is slowly varying at infinity and $\Phi(\infty)=\infty$, then
$$\limsup_{t \downarrow 0} \frac{S(t)  \Phi(t^{-1} \log |\log \Phi(1/t)|  )}{ \log |\log \Phi(1/t)|}=1,  \qquad  \mbox{almost surely.}$$
\end{lemma}

In our case $\xi_t=  \sum_{s<t} \sigma(e_s)$ and for the Laplace exponent $\Phi$ of the right-continuous generalised inverse of the local time we have  
$$\begin{aligned}
\exp\{-t \Phi(\lambda)\}
& = \E\big\{\exp\{ - \lambda \sum_{\heap{(s,e_s)\in Y}{s<t}} \sigma(e_s)\}\big\}\\
& = \exp\Big\{ - t \int dH^x(e) \big(1-{\mathrm e}^{-\lambda \sigma(e)}\big) \Big\},
\end{aligned}$$
where  we used 
(3.29) in Kingman~(1993). Since 
\begin{align}\label{grossphi}
\Phi(\lambda)= \int dH^x(e) \, \big(1-{\mathrm e}^{-\lambda \sigma(e)}\big) = \int_0^1  \, d t \, H^x\big( \sigma(e) > \sfrac{-\log t}{\lambda} \big),
\end{align}
to get the asymptotics of $\Phi(\lambda)$ as $\lambda\uparrow\infty$ we first look at  $H^x(\sigma(e)>\theta)$ as 
$\theta\downarrow 0$.

\begin{lemma}
We have
$$H^x(\sigma(e)>\theta)= a \int_{\skrib \setminus \{x\}} h(\xi) \;  p_\theta^{\skrib}(x,\xi) \;  d \xi.\\[-3mm]$$
\end{lemma}
 
\begin{proof} We observe that
$$\begin{aligned}
H^x(\sigma(e)>\theta )&=\int_{\skrib \setminus \{x\}} H^x(X_\theta \in  d \xi)\\
&=\int_{\skrib \setminus \{x\}} \lim_{\epsilon \downarrow 0} H^x(X_{\varsigma(\epsilon)+\theta} \in  d \xi; \; \varsigma(\epsilon)+\theta<\sigma),
\end{aligned}$$
where $\varsigma(\epsilon)$ is the first hitting time of $\partial \skrib(x,\epsilon)$. By the strong Markov property of the excursion measure, 
$$\begin{aligned}
H^x(\sigma(e)>\theta)&=\int_{\skrib \setminus \{x\}} \lim_{\epsilon \downarrow 0} H^x(   \prob^{X_{\varsigma(\epsilon)}}_h(X_\theta \in  d \xi; \; \theta <\sigma   ) ; \; \varsigma(\epsilon)<\sigma  )\\
&\le \int_{\skrib \setminus \{x\}} \lim_{\epsilon \downarrow 0} \sup_{z \in \partial \skrib (x,\epsilon)}  \prob^{z}_h(X_\theta \in  d \xi) \; H^x(\varsigma(\epsilon)<\sigma  )\\
&= \int_{\skrib \setminus \{x\}} \lim_{\epsilon \downarrow 0} \sup_{z \in \partial \skrib (x,\epsilon)}   \frac{h(\xi)}{h(z)}   p_\theta^{\skrib}(z,\xi)  \; H^x(\varsigma(\epsilon)<\sigma  )
 \;  d \xi,\\
\end{aligned}$$
and in view of  \eqref{h} and (\ref{L5.1(i)}), we have
$$\begin{aligned}
H^x(\sigma(e)>\theta)&\le \int_{\skrib \setminus \{x\}} h(\xi) \lim_{\epsilon \downarrow 0} \sup_{z \in \partial \skrib (x,\epsilon)}  p_\theta^{\skrib}(z,\xi)  \;  \frac{ H^x(\varsigma(\epsilon)<\sigma)}{|\log \epsilon |}  \;  d \xi\\
&= a \int_{\skrib \setminus \{x\}} h(\xi) \;  p_\theta^{\skrib}(x,\xi) \;  d \xi.
\end{aligned}$$
Analogously, we get
$$\begin{aligned}
H^x(\sigma(e)>\theta)&\ge \int_{\skrib \setminus \{x\}} \lim_{\epsilon \downarrow 0} \inf_{z \in \partial \skrib (x,\epsilon)}  \prob^{z}_h(X_\theta \in  d \xi) \; H^x(\varsigma(\epsilon)<\sigma  )\\
&= a \int_{\skrib \setminus \{x\}} h(\xi) \;  p_\theta^{\skrib}(x,\xi) \;  d \xi. \\[-9mm]
\end{aligned}$$
\end{proof}

The following lemma, see Port and Stone~(1978), Proposition 4.1,  will be used to replace $p^\ball_\theta$ by the Brownian transition
kernel in the main part of the integral.
\smallskip

\begin{lemma} \label{P-S}
Let $U$ be an open set and $x \in U$. Then there is a $\delta_0>0$ such that $\lim_{t \downarrow 0} p^{U}_t(\zeta,\xi)/p_t(\zeta,\xi)=1$ uniformly for 
all $\zeta,\xi \in \skrib(x,\delta_0)$.
\end{lemma}

We obtain the following asymptotics for  $H^x(\sigma(e)>\theta)$ as 
$\theta\downarrow 0$.

\begin{lemma}
$$H^x(\sigma(e)>\theta) \sim  a\, \sqrt{\frac{\pi}{2} \,} \log(1/ \theta)  \mbox{ as } \theta\downarrow 0.$$
\end{lemma}

\begin{proof}
For any small $\delta>0$ we have
$$\begin{aligned}
H^x(\sigma(e)>\theta)&= a \int_{\skrib \setminus \{x\}} h(\xi) \;  p_\theta^{\skrib}(x,\xi) \;  d \xi \leq a \int_{\skrib(x,\delta) \setminus \{x\}} h(\xi) \;  p_\theta^{\skrib}(x,\xi) \;  d \xi + \mathrm{const},
\end{aligned}$$
since the integral over the complement of $\skrib(x,\delta)$ is bounded. 
In view of \eqref{h} and Lemma~\ref{P-S}, for every $\eps >0$, we find a small $\delta>0$  and $d=d(\delta)>0$ such that, for~$\theta<d$,
$$\begin{aligned}
H^x ( \sigma(e) > \theta)&\le (1+ \eps) \, a  \int_{\skrib(x,\delta) \setminus \{x\}}\big|\log |x-\xi|\big| \;  p_\theta(x,\xi) \;  d \xi + \mathrm{const}\\
&= (1+ \eps)\, a \, \sqrt{2\pi}  \, \frac{1}{\theta}\,   \int_0^\delta  (-\log r)  \;  \exp\{- r^2/ 2 \theta \} \; r \; d r + \mathrm{const},
\end{aligned}$$
and, changing variables~ $s^2=r^2/ \theta $, we get
$$\begin{aligned}
H^x ( \sigma(e) > \theta)&\le (1+ \eps) \, a \, \sqrt{2\pi}  \, \left[  \int_0^{\delta/ \sqrt{\theta}}  (-\log s)  \;  \exp\{- s^2/ 2 \} \; s \; d s \right. \\
&\hspace{0.4cm}- \left. \log \sqrt{\theta} \int_0^{\delta/ \sqrt{\theta}} \exp\{- s^2/ 2 \} \; s \; d s \right]+ \mathrm{const}\\
&\le  - (1+ \eps)\, a \, \sqrt{2\pi}  \, \log \sqrt{\theta} + \mathrm{const}.
\end{aligned}$$
Similarly, for the lower bound we have
$$\begin{aligned}
H^x ( \sigma(e) > \theta)
&\ge - (1- \eps)\, a \, \sqrt{2\pi}  \, (1-{\mathrm e}^{-\frac{\delta^2}{2 \theta}})   \log \sqrt{\theta}+ \mathrm{const},\\
\end{aligned}$$
which completes the proof.\end{proof}

We now derive the asymptotics of $\Phi(\lambda)$ as $\lambda\uparrow\infty$ from~\eqref{grossphi}. 
For every fixed $d>0$ we have
$$\begin{aligned}
 \int_{{\mathrm e}^{-\lambda d}}^{1} d t \; H^x\big( \sigma(e) > \sfrac{-\log t}{\lambda} \big)
  \leq \Phi(\lambda)
 \le \mathrm{const}+\int_{{\mathrm e}^{-\lambda d}}^{1} d t \; H^x\big( \sigma(e) > \sfrac{-\log t}{\lambda} \big),
\end{aligned}$$
where the constant depends on $d$ but not on $\theta$.
Hence, for every $\eps>0$ there exists a small~$d>0$ such that
$$\begin{aligned}
-(1-\eps) \, a \, \sqrt{\sfrac{\pi}{2}}  \int_{{\mathrm e}^{-\lambda d}}^{1} \log (- \sfrac{\log t}{\lambda})  \, dt \leq 
\Phi(\lambda)&\leq \mathrm{const} -(1+\eps) \, a \, \sqrt{\sfrac{\pi}{2}}  \int_{{\mathrm e}^{-\lambda d}}^{1} \log (- \sfrac{\log t}{\lambda})  \, dt,
\end{aligned}$$
and, as $\log\log(1/t)$ is integrable over $(0,1)$, we get 
$$\Phi(\lambda)  \sim a \, \sqrt{\sfrac{\pi}{2}}\, \log \lambda \mbox{ as $\lambda\uparrow\infty$,}$$ 
which implies that $\Phi$ is slowly varying at infinity, and satisfies $\Phi(\infty)=\infty$.
\pagebreak[3]


We can now use Lemma \ref{F-P} and get \eqref{LIL} for the gauge function
$$\varphi(\eps)=a \, \sqrt{\sfrac{\pi}{2}}\,\frac{\log| \log \Phi(1/\eps)|}{\Phi(\eps^{-1} \log |\log \Phi(1/\eps)|)}\sim  
 \frac{\log\log \log(1/\eps)}{\log(1/\eps)}.$$

\section{Proof of the upper bound}

Fix a ball $\ball(0,R)$, $R>1$, and stop Brownian motion at the first exit time~$\tau$
from this ball. Given a cube $Q$ of side length~$r$ inside this ball, we define recursively
\begin{eqnarray*} 
\tau_1^Q&=&\inf\big\{t\ge 0 \, :\, B_t\in Q\big\},\\
\tau_{k+1}^Q&=&\inf\big\{t\ge \tau_k^Q+r \, :\, B_t\in Q\big\},
\end{eqnarray*}
with the usual convention that $\inf \emptyset=\infty$. 

\begin{lemma}\label{Kauftauk}
There exists $0<\theta<1$ such that for any $z\in\ball(0,R)$ and $0<r<\frac12$, 
$$\prob^z\big(\tau_k^Q<\tau \big)\le \theta^{k}.$$
\end{lemma}

\begin{proof}
It suffices to bound $\prob^z(\tau_{k+1}^Q\ge\tau\, |\, \tau_k^Q<\tau)$
from below by
$$\prob^z \big( \tau_{k+1}^Q\ge\tau \, \big|\, |B_{\tau^Q_k+r}-x|>2\sqrt{r}, 
\tau_k^Q<\tau \big) \, \prob^z \big( | B_{\tau^Q_k+r}-x|>2\sqrt{r} \, 
\big| \, \tau_k^Q <\tau \big), $$
where $x$ is the centre of~$Q$.
The second factor  can be bounded
from below by a constant not depending on~$r$. The first factor is bounded from below by 
the probability that planar Brownian motion started at any point
in $\partial \ball(0,2\sqrt{r})$ hits $\partial \ball(0,2R)$ before 
$\partial \ball(0,r)$. This  probability is given by
$$\frac{\log 2\sqrt{r} -\log r}{\log 2R - \log r} \ge 
\frac{-\frac12 \log r}{-\log r + \log 2R },\phantom{\Bigg\}}$$
which is bounded from zero by a positive constant independent of $r$.
\end{proof}

\begin{lemma}\label{Kaufcor}
Let ${\mathfrak C}_m$ be the set of dyadic cubes of side length $2^{-m}$
inside a fixed unit cube $U\subset\ball(0,R)$. Almost surely there exists a (random) 
integer $C$ so that for all $m\geq 1$ and cubes $Q\in {\mathfrak C}_m$ 
and $r=2^{-m}$ we have $\tau_{mC}^Q>\tau$.
\end{lemma}

\begin{proof} From Lemma~\ref{Kauftauk} we get that, for any positive integer $c$,
$$\sum_{m=1}^\infty
\sum_{Q\in {\mathfrak C}_m}\prob\big( \tau^Q_{ cm }<\tau \big)
\le \sum_{m=1}^\infty 2^{2m} \theta^{cm}.$$
Now choose $c$ so large that $4\theta^c<1$. Then, by the Borel-Cantelli lemma, for
all but finitely many $m$ we have $\tau^Q_{cm}\geq \tau$ for all 
$Q\in {\mathfrak C}_m$. Finally, we can choose a random $C>c$ to handle the 
finitely many exceptional cubes.
\end{proof}

To complete the proof we note that, on the event in the lemma, for a given $m$ we can cover
any set $\{ 0<t<\tau \colon B_t\in Q\}$, $Q\in{\mathfrak C}_m$,
with no more than $Cm$ intervals of length $r=2^{-m}$. This implies that, for any $z\in U$,
$${\mathcal H}^\phi\big(\{ 0<t<\tau \colon B_t=z\}\big) 
\leq \lim_{m\to\infty} Cm \phi(2^{-m}) =0,$$
under the assumption on $\phi$. Theorem~\ref{upper} follows as $U\subset\ball(0,R)$
and $R>1$ were arbitrary.

\section{Outlook}

As mentioned in the introduction, the strong form of Taylor's Problem~5 remains unresolved in this paper. 
Our work however allows us to make a conjecture. We believe
that our lower bound is sharp, or in other words that for every gauge function $\phi$ with 
$\phi(\eps)/\varphi(\eps) \to 0$ where $\varphi$ is as in Theorem~\ref{upper}, almost surely, 
$${\mathcal H}^\phi\big(\{t\geq 0 \colon B_t=x\}\big)=0\qquad \mbox{for every $x\in\R^2$.} $$
The reason for this belief is that our upper bound uses coverings of the level sets  by intervals 
of fixed size. If we were able to adapt the size of the covering intervals to the fluctuations of the
local time, one would expect a gain of order $\log\log\log(1/\eps)$, similar to our lower bound. 
However, the technical difficulties in carrying out such an estimate appear to be considerable and we 
therefore defer verification of our conjecture to future work.


\bigskip

\begin{thebibliography}{BBK94}

\bibitem[BBK94]{BBK94}
{\sc Bass, R.F., K.~Burdzy} and {\sc D.~Khoshnevisan}
\newblock Intersection local times for points of infinite multiplicity. 
\newblock {\em Ann. Probab.} 22:566--625 (1994).
\smallskip

\bibitem[BC95]{BC95}
{\sc Bertoin, J.} and {\sc M.E.~Caballero}
\newblock On the rate of growth of subordinators with slowly varying Laplace exponent. 
\newblock In: {\em S\'eminaire de Probabilit\'es XXIX}, Springer Lecture Notes in Mathematics Vol~1613, pp 125--132 (1995).
\smallskip


\bibitem[B87]{B87}
{\sc Burdzy, K.}
\newblock {\em Multidimensional Brownian excursions and potential theory}.
\newblock Longman, New York, 1987.
\smallskip


\bibitem[DEK58]{DEK58}
{\sc Dvoretzky, A., P.~Erd\H{o}s} and {\sc S.~Kakutani}
\newblock Points of multiplicity $\mathfrak{c}$ of planar Brownian motion.
\newblock \emph{Bull. Res. Council Israel} {\bf 7}~(1958) 157--180.
\smallskip

\bibitem[FP71]{FP71}
{\sc Fristedt, B.E.} and {\sc W.E. Pruitt}
\newblock Lower functions for increasing random walks and subordinators. 
\newblock {\em Z. Wahrscheinlichkeitstheorie} 18 (1971) 167--182.
\smallskip



\bibitem[K93]{K93}
{\sc Kingman, J.F.C.}
\newblock {\em Poisson Processes}.
\newblock Oxford University Press, Oxford, 1993.
\smallskip


\bibitem[LG87]{LG87}
{\sc Le Gall, J.-F.} 
\newblock Le comportement du mouvement brownien entre les deux instants o\`u
il passe par un point double.
\newblock {\em J. Funct. Anal.} {\bf 71}~(1987) 246--262.
\smallskip

\bibitem[MP10]{MP10}
{\sc M\"orters, P.} and {\sc Y.~Peres}
\newblock {\em Brownian motion}.
\newblock Cambridge University Press, Cambridge, 2010.
\smallskip

\bibitem[PS78]{PS78}
{\sc Port, S. C.} and {\sc C. J. Stone}
\newblock {\em Brownian motion and classical potential theory}.
\newblock Academic Press, New York, 1978.
\smallskip

\bibitem[Ta86]{Ta86}
{\sc Taylor, S.J.}
\newblock The measure theory of random fractals.
\newblock {\em Math. Proc. Cambridge Phil. Soc.} 100 (1986) 383--406.
\smallskip


\bibitem[PT87]{PT87}
{\sc Taylor, S.J.} and {\sc E. A. Perkins}
\newblock Uniform measure results for the image of subsets under Brownian motion. 
\newblock {\em Probab. Theory Rel. Fields} 76 (1987) 257--289.
\smallskip


\bibitem[Xi04]{Xi04}
{\sc Xiao, Y.}
\newblock Random fractals and Markov processes. %
\newblock In: \emph{Fractal Geometry and Applications: A Jubilee of Beno\^\i t Mandelbrot.}
\newblock Proc. Symposia Pure Math. {\bf 72}, 261--338, American Mathematical Society, Providence, 2004.
\smallskip

\end{thebibliography}
\end{document}